\documentclass[12pt,a4paper]{article}
\usepackage{amssymb}
\usepackage{amsmath,amsfonts,amsthm,color}
\usepackage{graphicx}
\usepackage{amsmath}
\usepackage{amsthm,amscd}
\usepackage{graphicx}
\usepackage{amsmath}
\usepackage{amstext}

\newtheorem{theorem}{Theorem}

\newtheorem{definition}[theorem]{Definition}

\title{Ergodic Optimization and Ground States: a brief Introduction}
\author{Artur O. Lopes\\
\smallskip  Dedicated to the memory of Gonzalo Contreras
 }

\begin{document}

\maketitle
 
\begin{abstract} 
Our goal in this short note is to briefly and succinctly describe some basic concepts and properties of Ergodic Optimization for readers unfamiliar with the subject. We avoid technical issues in order to provide a global overview of this topic. We will not attempt to cover all of the many contributions of various authors, who have greatly enriched the theory with invaluable results. The author has made a personal selection of the topics to be addressed, keeping in mind two main objectives: to motivate the reasons for studying the subject, and to describe schematically and pictorially its relationship with relevant concepts and properties of Statistical Mechanics, which is one of the sources of inspiration for the theory. We will not present new results or detailed proofs. Some examples will be provided. We describe some procedures that may  help in obtaining explicit solutions. We present some references that by no means aim to exhaust the bibliography on the subject,  where possible, minimizing the number of references.
 
	\end{abstract}
 
\section{Introduction}

Consider the set $\Omega =\{1,2,...,d\}^\mathbb{N}$ and the shift $\sigma$ acting on $\Omega$:  
\smallskip

$\,\,\,\,\,\,\,\,\,\,\,\,\, \,\,\,\,\,\,\,\,\,\,\,\,\,\,\,\, \,\,\,\sigma(x)=\sigma(x_1,x_2,...,x_n,...)= (x_2,x_3,...,x_n,...)$.
\medskip

The metric on $\Omega$ is:
\medskip
 
$\,\,\,
d (x,y)=\left( \frac{1}{2}\right) ^{\min \{n:
x_{n}\neq y_{n}\}},\,$ where\, $ x=(x_{1},x_{2},\ldots )\, \text{and}\,y=(y_{1},y_{2},\ldots ).$
\medskip

A Borel probability $\mu$ is $\sigma$-invariant if for any continuous function $f: \Omega \to \mathbb{R}$ we have
\begin{equation} \label{ff} \int (f \circ \sigma) d \mu = \int f d\mu.
\end{equation}

Equivalently, $\mu$ is $\sigma$-invariant if for any Borel set $E\subset \Omega$, we have that
 
\begin{equation} \label{EE}\mu(\sigma^{-1} (E))=\mu(E).
\end{equation} \
 
Denote by $\mathcal{M}_\sigma$ the set of $\sigma$-invariant probabilities over the Borel sigma-algebra.
 
When considering the iteration $\sigma^n(x)$, the $n$ does not represent time, but rather the translation in the lattice $\mathbb{N}$. One can think of $\{1,2,...,d\}$ as the set of possible spins located at each site $k$ of the lattice $\mathbb{N}$.

From the point of view of Statistical Mechanics,
$\mu\in \mathcal{M}_\sigma$ means that $\mu$ is in thermodynamic equilibrium.

A cylinder set of size $r$ is a set of the form $\overline{a_1,a_2,..,a_r} \subset \Omega$, given by
$$\overline{a_1,a_2,..,a_r} = \{(a_1,a_2,..,a_r,x_{r+1}, x_{r+2},..) \,| \text{where}\, x_j\in  \{1,2,..,d\}, j \geq (r+1)\}.$$
 
The set of all cylinders of all sizes $r\in \mathbb{N}$ generates the Borel sigma-algebra on $\Omega$. 
 
\begin{definition} \label{BKL0}
The Kolmogorov entropy of a $\sigma$-invariant probability $\mu$ is given by
$$ h(\mu)=$$
\begin{equation} \label{utum}  -\,\lim_{m \to \infty} \frac{1}{m}   \sum_{ (a_1,a_2,..,a_m) \in \{1,2,...,d\} ^m}  \mu(\overline{a_1,a_2,..,a_m})  \log (\mu(\overline{a_1,a_2,..,a_m})).
\end{equation}
\end{definition}

For general properties of entropy see \cite{Walters} or \cite{Ma}.
 
\smallskip
 
In Statistical Mechanics, a function 
$H:\Omega \to \mathbb{R}$ describes the interaction among the sequences $(x_1,x_2,...,x_n,...)$ and is usually called a Hamiltonian.

The relation between problems in the symbolic space
$\{1,2,...,d\}^{\mathbb{Z}}$ and those in $\{1,2,...,d\}^\mathbb{N}$ can be properly described via Proposition 1.2 in \cite{PP}, or via Appendix section 9 in \cite{Lo1}. Ergodic results for probabilities acting on the space $\{1,2,...,d\}^{\mathbb{Z}}$
can be deduced from the use of the Ruelle operator (a powerful tool) acting on functions defined on the space $\Omega=\{1,2,...,d\}^\mathbb{N}.$ This is not an issue here.
 
\section{Thermodynamic Formalism}
 
\smallskip
 
General references for Thermodynamic Formalism are \cite{PP}, \cite{LMMS}, \cite{Lo1}, \cite{KGLM} and \cite{Walters}.
 
Given a H\"{o}lder (sometimes merely continuous) function $$A:\Omega= \{1,2,...,d\}^\mathbb{N} \to \mathbb{R},$$ in this section we are interested in the study of the set of probabilities $\mu$ maximizing the Topological Pressure:
\begin{equation} \label{AA}P(A):= \sup_{\rho \in \mathcal{M}_\sigma} \left( \int A \,\,d \rho+h(\rho )\right),
\end{equation} \label
where $h(\rho)$ is the Kolmogorov entropy of $\rho$.

One is interested in probabilities $\mu_A$ which maximize $P(A)$, that is, those satisfying
\smallskip
 
$\,\,\,\,\,\,\,\,\,\,\,\,\,\,\,\,\, \,\,\,\,\,\,\,\,\,\,\,\,\,\,\,\,\,\, \,\,\,\,\,\,\,\,\,\,\,\,\,\, {P(A)=  h(\mu_A) + \int A \, d \mu_A.}
$
\medskip

Any such probability $\mu_A$ is called an equilibrium state (or measure) for $A$.

 The study of \eqref{AA} for the case of functions $A$ that are merely continuous (see \cite{Walters}) presents some conceptual differences from the setting where $A$ is of H\"{o}lder class (see \cite{PP}). Here we are interested in the latter case.
 
In the above setting, a function $A$ is usually called a potential. In accordance with Statistical Mechanics, $A$ should represent $-H=A$, where $H$ is the Hamiltonian. The minus sign is related to the fact that for a probability $\mu$ in equilibrium (under the action of the Hamiltonian $H$), states of high energy are less probable.

By equilibrium under the action of the Hamiltonian $H$ we mean the one associated with
\begin{equation} \label{HH9}P(-H):= \sup_{\rho \in \mathcal{M}_\sigma} \left(- \int H \,\,d \rho+h(\rho )\right).\end{equation}

Given $H$, if there exists more than one equilibrium state for \eqref{HH}, we say that the Hamiltonian $H$ exhibits a phase transition.
\medskip
 
In Statistical Mechanics, if $H:\Omega \to \mathbb{R}$ is the Hamiltonian, then $\int H d \mu$ is the mean energy of the state $\mu \in \mathcal{M}( \sigma)$.
\medskip

{\bf Second Law of Thermodynamics:} There is a tendency in nature for equilibrium to be reached in states of higher entropy. This fact can be formalized as follows: for a system under the action of Hamiltonian $H$, given a real value $E$ (representing energy), consider the set
$$\mathfrak{E}_E= \{ \mu\in \mathcal{M}_\sigma \,|\,  \int H d \mu =E\}.$$
 
Under the above conditions, the probability at equilibrium is the one maximizing $h(\mu)$ among elements $\mu \in \mathfrak{E}_E$.
 
\medskip
 
If the function $A$ is of H\"older class, the $\sigma$-invariant probability $\mu_A$ is {unique}; it has positive entropy, is ergodic, mixing, and positive on cylinder sets (see \cite{PP}). The {Ruelle Theorem} is a helpful tool for obtaining its ergodic properties (see \cite{PP}, \cite{Lo1} and \cite{Bala}). Explicit examples of equilibrium probabilities for some nontrivial potentials are presented in \cite{CDLS}.
 
\medskip

When $A=0$, we obtain the equilibrium probability $\mu_0$, which is the maximal entropy measure: the independent probability with weights $\frac{1}{d}$, and $P(0)= \log d.$
 
\medskip

A large value of the entropy of a $\sigma$-invariant probability $\mu$ corresponds to a large randomness of the samples; more uncertainty. Our dynamical entropy $h(\mu)$, more than concave, is affine (see \cite{Walters}).

In agreement with Statistical Mechanics --- in the one-dimensional lattice $\mathbb{N}$ --- it is natural to introduce a parameter $\beta>0$, where {$\beta=\frac{1}{T}$, and $T>0$ represents temperature; and to consider $P(\beta \,A)$}.
\medskip

Via the Legendre transform one can relate the maximizing solutions of the Second Law problem (described above) with the maximization of pressure as in \eqref{HH}; the so-called MaxEnt method (see Appendix 9.2 in \cite{Lo1} and \cite{Lal2}).
 
In this case, in consonance with the {Second Law of Thermodynamics}, if {$H$ is a Hamiltonian $H:\Omega \to \mathbb{R}$} (which can represent energy), the probability {$\mu_{- \beta\, H}\in \mathcal{M}_\sigma$ is the equilibrium probability for the system under the influence of $H$, at temperature $T= \frac{1}{\beta}$}. Then, in our notation, $- \beta\, H=\beta A$.

More precisely, in accordance with Statistical Mechanics, we will be interested in equilibrium probabilities $\mu_{ - \frac{1}{T}\,H }$ for
\begin{equation} \label{HH}P(- \frac{1}{T}\,H):= \sup_{\rho \in \mathcal{M}_\sigma} \left(-\frac{1}{T}\, \int H \,\,d \rho+h(\rho )\right).\end{equation}
\medskip

Consider the spin lattice $\{+,-\}^\mathbb{N}$. The spin up is associated with $+$ and spin down with $-$.
\medskip
 
A measurement could result in $(+,-,-,+,+,-,-,...)$, which is a random sample. Consider a physical system at positive temperature $T$ subject to a Hamiltonian $H:\{+,-\}^\mathbb{N}\to \mathbb{R}$. 
In this case one is interested, for instance, in the {probability $\mu_{-\frac{1}{T}\,H}$ of the cylinder set $\overline{+,-,++} \subset \{+,-\}^\mathbb{N}$}, under the equilibrium regime for $H$ and temperature $T>0$; this probability $\mu_{-\frac{1}{T}\,H}$ provides statistics of samples. If $\mu_{-\frac{1}{T}\,H}(  \overline{+,-,++} )$ is very small, this means that the arrangement $+,-,++$ is quite unlikely to be observed in the first 4 sites of the lattice $\mathbb{N}.$

\medskip
 
The probability $\delta_{+,+,+,+,...}= \delta_{+^{\infty}}$ represents a type of magnetization (all spins up) on the lattice $\mathbb{N}.$ There is no randomness in this case. The entropy of the $\sigma$-invariant probability $\delta_{+^{\infty}}$ is zero.

\smallskip
 
$T=0$ means absolute zero (around $-273$ degrees Celsius). In real-world experiments, the temperature of the system under consideration can be lowered to approach zero temperature.

It is known from experiments in Physics that when {decreasing the temperature $T \to 0$} of a metal, the randomness of the corresponding equilibrium state $\mu_{-\frac{1}{T}\,H}$ also decreases. In most cases, when $T\to 0$,  there is a tendency toward decreasing entropy and the metal tends to exhibit magnetic properties. How can this be described in mathematical terms?
 
\smallskip

For low temperatures, the spins of atoms at each site become aligned with each other, and this can be properly described as {something like $\delta_{+,+,+,+,...}= \delta_{+^{\infty}}$} (a {\it ``non-random''} probability), or alternatively $\delta_{-,-,-,-,...}= \delta_{-^{\infty}}$.
 
\smallskip

Another possibility of magnetization could be $\frac{1}{2} (  \delta_{+,+,+,+,...} + \delta_{-,-,-,-,...});$ or, more generally, a  $\sigma$-invariant probability supported on a periodic orbit (for the shift $\sigma$).

\medskip
 
 \section{Ergodic Optimization}

General references on the topic are \cite{BLL}, \cite{Ga} and \cite{Lo1}.
\smallskip

A continuous function $A:\Omega\to \mathbb{R}$ can be seen as a {cost, or profit}, and {$\int A d \mu$}, for $\mu\in \mathcal{M}_\sigma$, represents the (dynamical) {$\mu$-mean value of the cost}. One could consider minimization or maximization of $\int A d \mu$ as well; simply replace $A$ by $-A$. Given $A:\Omega \to \mathbb{R}$, maximizing $\mu \to \int A d \mu$ is the main problem in Ergodic Optimization.

  An extensive bibliography on Ergodic Optimization appears in \cite{J2}.
 The study of the case of potentials $A$ that are merely continuous (see \cite{J1}, \cite{J6} and \cite{BZ}) presents some conceptual differences from the setting where $A$ is of H\"{o}lder class (see \cite{BLL}, \cite{Ga} and \cite{CLT}). Here we are interested in the latter case.

\smallskip

 
 Given a continuous function $A:\Omega\to \mathbb{R}$, consider the problem:
$$ \alpha(A):=\sup_{\rho \in\cal{M}_{\sigma}}  \{\int A(x) d \rho(x)\}.
$$
 {$\alpha(A)$ is called the maximizing value of $A$}.

\medskip
In Ergodic Optimization one is interested in probabilities {$\mu^A$ which maximize $\alpha(A)$}, that is, those satisfying
\smallskip
 
$\,\,\,\,\,\,\,\,\,\,\,\,\,\,\,\,\, \,\,\,\,\,\,\,\,\,\,\,\,\,\, \,\,\,\, \,\,\,\,\,\,\,\,\,\,\,\,\,\,\alpha(A)=   \int A \, d  {\mu^A}$
\smallskip

 Any such probability $\mu^A$ is called a maximizing probability for $A$. 
If the potential $A$ is of H\"older class, the maximizing probability $\mu^A$ need not be unique (see Section 6 in \cite{Lo1}).
 
Several explicit examples illustrating the theory are presented in \cite{FeLoEl1}.
 
\smallskip
 
If $\mu^A$ maximizes the function $A$, it also maximizes $A$ plus any constant.

\smallskip

A natural question is: given a H\"older potential $A$, when is $\mu^A$ unique? Using the H\"older norm, uniqueness holds for a dense set of H\"older potentials --- adapting results from convex analysis (see \cite{CLT} for a proof).


For a fixed H\"older potential $A$, consider $\beta \to \infty$, the Pressure $P(\beta A)$ (which is convex and analytic in $\beta$, as proved in \cite{PP}), and the associated family of equilibrium probabilities $\mu_{\beta A}$, $\beta>0$. One can show that any weak limit $\mu$ of a subsequence $\lim_{\beta_n\to \infty} \mu_{\beta_n A}=\mu$ will be a maximizing probability for $A$ (see \cite{CLT}, \cite{BLL} or \cite{Lo1}). ; more results on the topic in Section \ref{Gr}.
\smallskip

   $\beta_n \to \infty$ means $T_n= \frac{1}{\beta_n}\to 0.$
Ergodic Optimization is the natural setting in which to analyze {equilibrium probabilities at zero temperature}.
\medskip

Given a $k$-periodic orbit $\{x, \sigma(x),\sigma^2(x),...,\sigma^{k-1}(x)\}$, where $\sigma^k(x)=x$, we say that $ \,\,\,\,\,\,\rho=\sum_{j=0}^{k-1}\frac{1}{k} \delta_{\sigma^j(x)}$
is the {$\sigma$-invariant probability associated with such an orbit}; it has zero entropy (see \cite{Walters}).

The set of $\sigma$-invariant probabilities associated with periodic orbits (of all periods) is {weakly dense} in the set of $\sigma$-invariant probabilities (see \cite{Sig}).
\medskip

  Given a continuous potential $A:\Omega \to \mathbb{R}$ and $\epsilon >0$, there exists a probability $\rho$ supported on a {periodic orbit} such that 
$\int A d \rho > \alpha(A)- \epsilon.$
 
\medskip

\textbf{Fourth Law of Thermodynamics:} (Nernst law) The {entropy vanishes} when there is
only one equilibrium state at absolute zero temperature.
\medskip
 
In the same spirit as this postulate, we emphasize the fact that the entropy of an invariant probability supported on a unique periodic orbit is zero.
\medskip

Another source of inspiration for Ergodic Optimization is the Aubry--Mather Theory (see \cite{Fathi} and \cite{CI}).
 
 \begin{theorem} \label{T1}   The analogue of the Ma\~n\'e Conjecture in Aubry--Mather Theory for the ergodic setting: the property that the maximizing measure $\mu^A$ is supported on a unique periodic orbit is dense in the set of H\"older potentials $A$.
 \end{theorem}
 This quite important result was proved by Gonzalo Contreras in 2016 in \cite{Co} (partial results in \cite{CLT});  it is somehow in consonance with Nernst law.

Related work appears  in \cite{BZ}, \cite{J6}, \cite{Morris2} and \cite{QS}.
 
\medskip

A continuous function $u:\Omega\to\mathbb{R}$ is called a subaction for $A$ if
 \begin{equation} \label{lyr}  u(\sigma(x))\geq   u(x)  +A(x) -\alpha(A) ,\,\,\;\; \text{ for all } \; x \in  \Omega.
\end{equation}

A continuous function $u:\Omega\to\mathbb{R}$ is called a {calibrated subaction} for $A$ if
\begin{equation} \label{lyr1}  u(y)=\max_{\sigma(x)=y} [u(x)+ A(x)-\alpha(A)],\,\,\;\; \text{ for all } \; y \in  \Omega.
\end{equation}

The above equation corresponds to the discrete-time version of the Lax--Oleinik semigroup equation used to investigate the Hamilton--Jacobi equation in Aubry--Mather theory (see \cite{GL1}).
 
\medskip

\begin{theorem} \label{T2} (see \cite{CLT}) \label{existsubac} Given a H\"older function $A:\Omega \to \mathbb{R}$, there exists a H\"older function $u$ which is a {calibrated subaction} for $A$.
\end{theorem}

In order to find an explicit subaction $u$, it is necessary to first guess the value $\alpha(A)$; this is not a straightforward question in general. We will address this issue later.
 
\medskip

If we add a constant $c$ to a calibrated subaction $u$ for the potential $A:\Omega \to \mathbb{R}$, we obtain another one for $A$. 
Given $A$, we say that the calibrated subaction is unique when it is unique up to the addition of a constant. When $\mu_A$ is unique there is only one calibrated subaction.
There exist H\"older potentials $A$ for which the calibrated subaction is not unique; in this case, one can use the {Ma\~n\'e potential} to find different subactions (see \cite{GL1} and \cite{BLLx}).
 
\medskip

If
$u:\Omega \to \mathbb{R}$ is a subaction for $A:\Omega \to \mathbb{R}$, it follows that 
 
 $$ R_u(x)=R(x)\,:=\, u(\sigma(x)) -  u(x) - A(x) +\alpha(A)\geq 0.$$
 
\smallskip

Given $u$, we will be interested in the points $x\in \Omega$ such that {$R(x)=0.$}
 
\medskip

We call the contact locus of a subaction $u:\Omega \to \mathbb{R}$ the set
\medskip
 
$ { \mathbb M_A(u) }:= \{x\in \Omega \,|\, A(x) - u \circ \sigma (x)  + u (x)) =\alpha(A)\}= \{x\in \Omega \,|\,R(x)=0\} .
$

\medskip
 
Determining the set of points $x$ where $R(x)=0$ is of great importance: it helps to localize the support of maximizing measures.

\begin{theorem} \label{T3} (see \cite{CLT} and \cite{BLL}) ---
Suppose $\mu$ is maximizing for $A$ and $u$ is a subaction. Then,
\begin{equation} \label{esqr27} (u \circ \sigma) (x)-  u (x)= A- \alpha(A),
\end{equation}
for all $x$ in the support of $\mu$; that is, $R(x)=0$.

 Moreover, if a probability $\mu\in \mathcal{M}_\sigma$ is such that \eqref{esqr27} holds for all points $x$ in its support, then $\mu$ is maximizing for $A$.
\end{theorem}
 
\smallskip

{\bf Claim:} The supports of maximizing probabilities are contained in $\mathbb M_A(u)$ for any subaction $u$.

\medskip
 
\begin{proof} It is known that $R(x)\geq 0$ for all $x$. Then, for $\mu$ maximizing $A$,
\medskip
 
$ \,\,\,\,\,\,\,\,\,\,\,\,\,\,\,\,\,\,\,\,\,\,\,\,\,\,\,\,\,\,\,\,\,\,\,\,\,\,\,0\leq \int R(x) d \mu(x)= $
\medskip
 
$ \,\,\int\, (u(\sigma(x)) -  u(x) - A(x) +\alpha(A)) d \mu =\int ( - A(x) +\alpha(A)) d \mu=0.$
 
Then $R(x)=0$ for $\mu$-a.e.\ point $x$.
\end{proof}
 
\medskip
 
Note, however, that even when $\mu^A$ is unique, given a calibrated subaction $u$ for $A$, the set {$\mathbb M_A(u)$ may strictly contain the support of $\mu^A$}.

If we interpret $x\to u(\sigma(x)) -  u(x)  =\mathfrak{d}(u)(x)$ as the {discrete time derivative} of the function $u$, the equality
\begin{equation} \label{esqr} \mathfrak{d}(u)(x)=u(\sigma(x)) -  u(x) = A(x) - \alpha(A),
\end{equation}
for $x$ in the support of $\mu^A$, can be rephrased as: $u$ is the {\bf primitive} of $A- \alpha(A)$ on the support of the maximizing probability $\mu^A$.

\smallskip
 
Using the $1/2$-algorithm described in \cite{FLO1} and \cite{FeLoEl1}, it is possible in some cases to guess {the explicit expression of a subaction for a given potential $A$; for instance, when $A$ is the indicator function defined on a cylinder set in $\{0,1\}^\mathbb{N}$. The expression for the implementation of the iterative procedure on  a desktop is in \cite{FLO1} and the Code in  Python is available in \cite{Fer}; the performance of the $1/2$-iteration method  is  analyzed in \cite{FLO1} .

 You do not need  to know $\alpha(A)$  to apply the $1/2$-method.  Uusing your desktop you derive explicit results which can be checked with pen and paper. 
  
The plan is to consider the associated binary expansion on the interval $[0,1]$, and from the pictures obtained from a certain iterative process to guess the solution: we will use the iterative $1/2$-method. The idea is first to identify elements  $w\in \Omega=\{0,1\}^\mathbb{N}$ with elements $x\in [0,1].$ We will elaborate on this.

\smallskip

Suppose the symbolic space is $\{0,1\}^\mathbb{N}$. One can visualize, via the binary expansion, elements $x=(x_1,x_2,x_3,..,x_n,..) \in \{0,1\}^\mathbb{N}$ as
\begin{equation} \label{soz} z= x_1 2^{-1}+ x_2 2^{-2} + x_3 2^{-3}+... +x_n 2^{-n}+...\in [0,1].
\end{equation}
 
Figure \ref{fig12i} illustrate this point of view; it helps to see geometric pictures.

We will explain how to use the procedure of guessing the subaction $u$ of a potential $A$ through examples. Take $A=I_{\overline{01}}$: the $1/2$-iteration method produces the picture described in Figure \ref{fig12ab} (it is not necessary many iterations). From this, a natural candidate is to set $u= - I_{\overline{0}} +1/2\, I_{\overline{1}}.$ We substitute this function $u$ into the subaction equation \eqref{lyr} and verify that it is indeed satisfied. By inspection we get that the value $\alpha(A)=1/2$, $R= 1/2\, I_{ \overline{00}} + 1/2 \,   I_{ \overline{11}}$  (which is is zero on the set
$\{(01)^\infty, (01)^\infty\}$); it follows from Theorem \ref{T3} that  $\frac{1}{2}\delta_{(01)^\infty} + \frac{1}{2} \delta_{(10)^\infty}$ is a maximizing probability for such $A$.     Thus, we obtain a rigorous mathematical result - it can be verified by hand - using the hint produced by the above-mentioned method obtained via computer. A similar procedure for finding $u$ for the potential $A=I_{\overline{01111}}$ is described in Figure \ref{purqt} (see also Example 37 in the Section 6.2 in \cite{Lo1}). We use the software Mathematica to generate the pictures.
 
\medskip

\begin{figure}[h!]
\centering
\includegraphics[scale=0.32,angle=0]{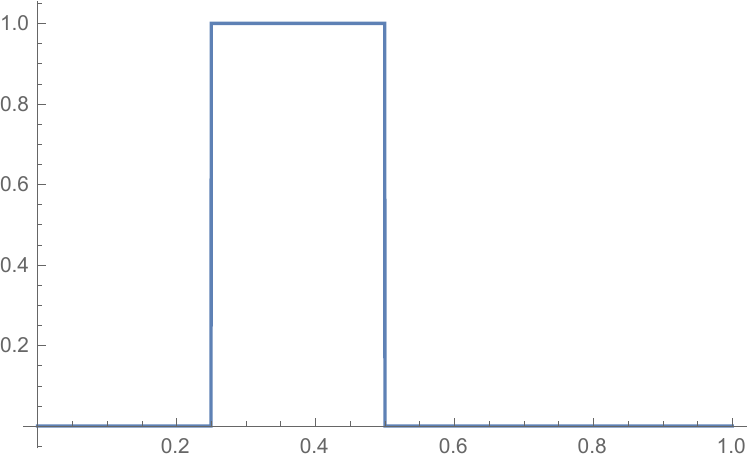}
\caption{The graph of $x\to A(x)$ ``defined''  for $x\in[0,1]$ corresponding to the potential $A=I_{\overline{01}}$.}
\label{fig12i}
\end{figure}

\begin{figure}[h!]
\centering
\includegraphics[scale=0.32,angle=0]{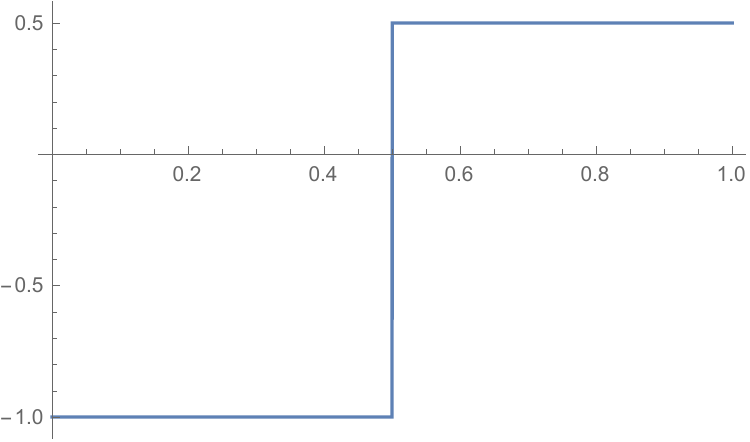}
\caption{The graph of  the subaction $x \to u(x)$ ``defined'' for $x\in[0,1]$ obtained from the $1/2$-iteration method when $A=I_{\overline{01}}$. From the computer-generated picture we guess that the subaction of $A$ is $u= - I_{\overline{0}} +1/2\, I_{\overline{1}}.$ One can verify rigorously (by hand) that such $u$ is indeed a solution to equation \eqref{lyr}. In this case  $\frac{1}{2}\delta_{(01)^\infty} + \frac{1}{2} \delta_{(10)^\infty}$ is the maximizing probability and $\alpha(A)=1/2$. Not many iterations were required to obtain this picture.}
\label{fig12ab}
\end{figure}

\begin{figure}[h!]
\centering
\includegraphics[scale=0.52,angle=0]{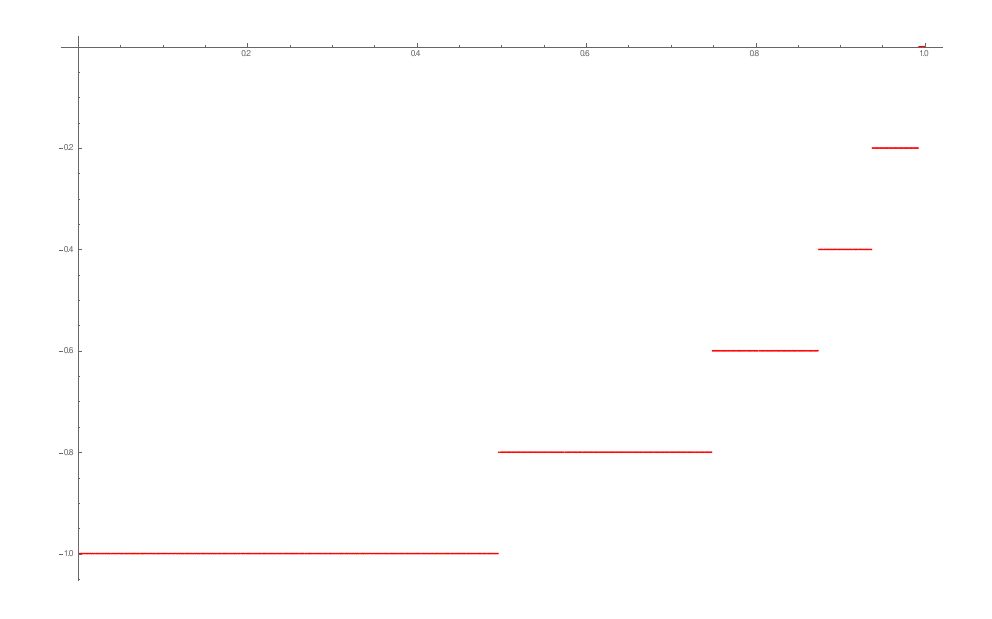}
\caption{The graph of $x \to u(x)$ on $[0,1]$ obtained from the $1/2$-iteration method when $A=I_{\overline{01111}}$. From the above picture one can guess the expression for a subaction $u$ in the case of this potential $A$. One can verify that such $u$ is a solution to the subaction equation \eqref{lyr}, where $\alpha(A)=0.2$.}
\label{purqt}
\end{figure}

\begin{figure}[h!]
\centering
\includegraphics[scale=0.4,angle=0]{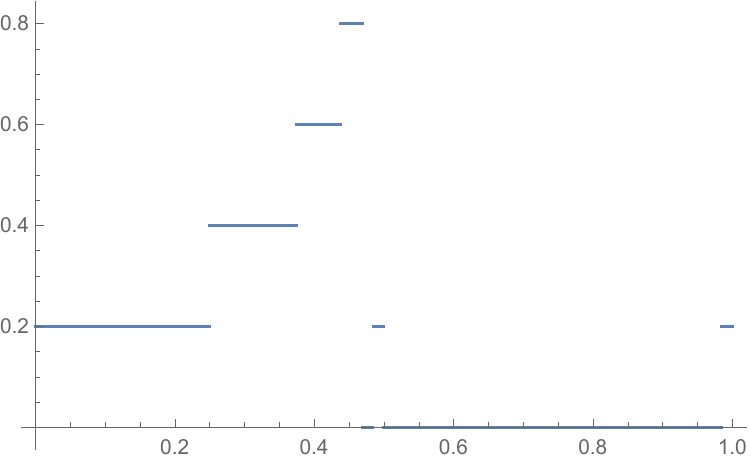}
\caption{ The graph on $[0,1]$ corresponding to the function $x \to R(x)$ for the potential $I_{\overline{01111}}$ (using  the subaction $u$ obtained from Figure \ref{purqt}). It is known that the function $R$ vanishes on the support of the maximizing probability. From the computer-generated function $R$, it is possible to see that the maximizing probability is $\mu^A=\frac{1}{5} \sum_{j=0}^{4} \delta_{\sigma^j (\, (01111)^\infty \,)}$. Note that the function $R$ can vanish outside the support of $\mu^A$. Also note that $\mu^A(\overline{1})>\mu^A(\overline{0})$.}
\label{purqt1}
\end{figure}

A point $ x \in \Omega$ is said to be {\it non-wandering with respect to $A:\Omega \to \mathbb{R}$}, if {for every $ \epsilon > 0 $, there exists an integer
$ k \ge 1 $ and a point $ y \in \Omega $} such that
$$ d(x, y) < \epsilon, \;\;\; d(x, \sigma ^k(y)) < \epsilon \; \text{ and } \;
\Big| \sum_{j = 0}^{k - 1} (A - \alpha(A)) \circ \sigma ^j (y) \Big| < \epsilon. $$

For practical purposes we can assume $\alpha (A)=0$.
In this case the above means that pseudo-orbits beginning at $x$ can have an arbitrarily small cumulative cost along the sequence.
 
\medskip

Up to resolution $\epsilon$, $\,\,\, \{y, \sigma(y),..., \sigma^{k-1}(y)\}$ {looks like a periodic orbit.}
\medskip

We denote by $ \Omega(A) $ the {\it set of non-wandering points} with respect to the observable $ A \in C^0(\Omega) $.
When the observable is H\"older, $\Omega(A)$ contains the support of all maximizing probability measures for $A$ (see \cite{CLT}, \cite{GL1} or \cite{Ga}).
 
\medskip

There is no mention of subactions in the above definition of $\Omega(A).$
 
\smallskip
 
We are interested in {finding $u$ so that $\Omega(A)=\mathbb{M}_A(u)$}.

A subaction $ u \in C^0(\Omega) $ is said to be separating (with respect to $A$) if it satisfies {$ \mathbb M_A(u) = \Omega(A) $; such $u$ provides the smallest possible $\mathbb M_A(u)$}.

\smallskip

The next two results were proved in \cite{GLT}:

\begin{theorem} \label{T4}
Given an $\alpha$-H\"older potential $ A : \Omega \to \mathbb R $,
{there exists an $\alpha$-H\"older separating subaction for $ A $}. Furthermore, in the $\alpha$-H\"older topology,
 the subset of $\alpha$-H\"older separating subactions is generic among all $\alpha$-H\"older
 subactions.
 \end{theorem}

\begin{theorem} \label{T5} The subset of Lipschitz continuous potentials $A$ such that for some subaction $u$, the set {$ \mathbb{M}_A(u)=\Omega(A)$} coincides with the {support of the unique (ergodic) maximizing probability} (with support in a periodic orbit) is {dense} in the Lipschitz topology.
\end{theorem}
 \medskip

 The following result is a consequence of the continuously varying support property (see \cite{CLT}).
 
\medskip

 \begin{theorem} \label{T8} In the case where the unique maximizing measure for $A$ is supported on a single periodic orbit,
there exists {a neighborhood $U$ of $A$} in the Lipschitz topology such that, for all $C$ in $U$, the maximizing probability for $C$ coincides with the maximizing probability for $A$. 
\end{theorem}

\section{Ground States} \label{Gr}
\medskip
 
Consider a variable parameter {$\beta \to \infty$}, and the associated family of equilibrium probabilities $\mu_{\beta A}$, $\beta=\frac{1}{T}>0.$
\smallskip

Any $\nu\in \mathcal{M}_\sigma$ such that, for some sequence $\beta_n$, we have the weak convergence
$$\lim_{\beta_n \to \infty} \mu_{\beta_n A}=\nu,$$
is called a ground state for $A$. One can show that such $\nu$ is maximizing for $A$ (see \cite{BLL}, \cite{Ga}, \cite{Le}, \cite{CoLe} or \cite{Lo1}).

\smallskip
When, for a given $A$ and $\nu$, we have
  {$\lim_{\beta \to \infty} \mu_{\beta A}=\nu,$}
we say that there exists {selection of a maximizing probability at zero temperature}. The measure $\nu$ is the selected probability. If the maximizing probability $\mu^A$ is unique, then $\nu =\mu^A$; a unique ground state.
 
\medskip

Suppose there exist two distinct maximizing probabilities $\mu_1,\mu_2$; they could both be ground states for $A$ (distinct examples appear in \cite{BLLx} and \cite{LeMe}). 
 
Different subsequences $\lim_{\beta_n\to \infty} \mu_{\beta_n A}$ could possibly converge to each of the two maximizing probabilities. Indeed, this can happen, as mentioned in \cite{CGU}.

\smallskip

In an independent work, when $\Omega=\{0,1\}^\mathbb{N},$ the authors of \cite{BGT} consider a certain subclass of potentials where the two maximizing probabilities for $A$ are $\mu_1=\delta_{0^\infty}$ and $\mu_2=\delta_{1^\infty}$; two magnetic states at zero temperature. In some of the examples, $A$ is of H\"older class.

\smallskip

\begin{theorem} \label{T10} By specifying certain explicit values of $A$ on $\overline{00^n1}$, $\overline{01^n0}$, $\overline{11^n0}$ and $\overline{10^n1}$, $n \in \mathbb{N}$, one can produce examples where
 
\medskip
 
a)\,$ \lim_{\beta\to+\infty} \mu_{\beta\, A} = \delta_{1^\infty}$,
 
or, 
 
\medskip
 
b) for an explicit value $c$, one has $\lim_{\beta \to + \infty }\mu_{\beta A} = \frac{1}{1+c^2}\delta_{0^\infty} + \frac{c^2}{1+c^2} \delta_{1^\infty},$
 
or,
\medskip
 
c) There exists a sequence $\beta_n \to \infty$ such that
 
$\lim_{k\to \infty}\mu_{\beta_{2k}\, A} = \delta_{0^\infty}$  and $\lim_{k\to \infty} \mu_{\beta_{2k+1}\, A} = \delta_{1^\infty}$.
 
\end{theorem}
 
\medskip
 
In another direction, in \cite{BLT} (see also \cite{Men} and \cite{LeMe}) for a more general setting) the authors prove the following result:
 
\begin{theorem} \label{T12} When, for a given $A$, we have
  {$\lim_{\beta \to \infty} \mu_{\beta A}=\mu^A$}, using properties of an associated subaction, one can establish the existence of a Large Deviation Principle (LDP) when {temperature goes to zero}: there exists a lower semicontinuous deviation function $I:\Omega \to [0,\infty]$ such that for any cylinder $C$,
\[
\lim_{\beta \to +\infty} \frac{1}{\beta} \log
\mu_{\beta \, A}(C)=-\inf_{ x \in C} I(x): = Q(C)\leq 0.
\]
 
When $\lim_{\beta \to \infty} \mu_{\beta A}(C)= \mu^A(C)=0$, $Q(C)<0$, we obtain the exponential rate of convergence as $\beta \to \infty$:
\medskip

$$ 0\swarrow\mu_{\beta A} (C)\sim e^{\,\,\beta Q(C)}$$
 
\medskip

One can show that 
$I(x)=0$ for $x$ in the support of $\mu^A$. Indeed, $I$ is given by 
\begin{equation} \label{ger}I(x)=\sum_{n\geq0}\big( u\circ\sigma-u-(A-\alpha(A))
\big)\circ\sigma^n(x)=\sum_{n\geq0} R(\sigma^n(x)) ,\end{equation}
where $u(x)$ is any calibrated subaction for $A$.

\end{theorem}
 
Expression \eqref{ger} further highlights the importance of the role of subactions.

 \section{Maximization under constraints}

For a H\"{o}lder (or continuous) $A$, the problem of finding maximizing probabilities for $A$ with {constraints} was considered in \cite{GL7}.
It is a kind of generalization of the concept of rotation number to the setting of symbolic dynamics (see also \cite{KW1} and \cite{KGLM}).

Let $ \varphi: \Omega  \to \mathbb R^n $ be a continuous map
with coordinate continuous functions $\varphi_1, \ldots, \varphi_n : \Omega  \to \mathbb R$; that is, {$\varphi=(\varphi_1, \ldots, \varphi_n)$}. We can define an {induced map $\varphi_*: \mathcal{ M}_\sigma \to \mathbb R^n $}
given by
 
$\,\,\,\,\,\,\,\,\,\,\,\,\,{\displaystyle \varphi_*(\mu) = \left(\int \varphi_1 \;
d\mu, \ldots, \int \varphi_n \; d\mu \right)}\in \mathbb{R}^n $.
 
\medskip

$\varphi_* $ is a continuous and affine map.
\medskip
 
We call {$ \varphi_*(\mu) $ the rotation vector of the measure $
\mu \in \mathcal M_\sigma $}.
\medskip
 
The image $
\varphi_*(\mathcal M_\sigma) \subset \mathbb R^n $ is a convex compact set. We call {$
\varphi_*(\mathcal M_\sigma) $ the rotation set of $\varphi$}.
\medskip

For $ h \in
\varphi_*(\mathcal M_\sigma) $, the fiber $\varphi_{*} ^ {-1}(h) $ is
called the rotation class of $ h $.
\medskip

$$\{\mu \in \mathcal M_\sigma\,|\, \varphi_* (\mu)=h\}= \varphi_{*} ^ {-1}(h)
\subset \mathcal{ M}_\sigma $$ is a convex compact set.

For $ A \in C^0(\Omega) $ and $\varphi$, we define the so-called beta function $ \beta_{A,
\varphi}: \varphi_*(\mathcal M_\sigma) \to \mathbb R $, such that for $h\in \mathbb{R}^n\in \varphi_*(\mathcal M_\sigma) $:
$$ \beta_{A, \varphi}(h) = \sup_{\mu \in \mathcal{M}_\sigma} \left \{\int A \,d\, \mu\, |\, \varphi_* (\mu)=h  \right \}. $$
\medskip

We call the function $ \varphi $ the constraint and
the function $ A $ the potential to be maximized.

\medskip
 
Given $A$, $\varphi$ and $h\in \mathbb{R}^n$,
we are interested in
{probability measures $\mu$ belonging to the rotation class of $ h $
that maximize the integral $ \int A d \mu$}.
 
\medskip

Consider the set
$$ \text{\Large $\mathit m$}_{A, \varphi}(h) = \left \{\mu \,|\,  \varphi_{*} (\mu)=h \,\,\text{and}\,\, \int A \; d\mu = \beta_{A, \varphi}(h) \right \}\neq \emptyset. $$
If {$ \mu \in \text{\Large $\mathit m$}_{A, \varphi}(h) $}, we say
that {$\mu $ is an $(A,\varphi, h)$-maximizing probability}.

The following results were proved in \cite{GL7}.
 
\smallskip

\begin{theorem} \label{T15} Given $\varphi$, 
consider a fixed $ h \in \varphi_* (\mathcal M_\sigma) $. There exists a {residual
subset} $\mathcal G = \mathcal G (h) \subset C^0(\Omega) $ such that,
for each potential $ A \in \mathcal G $, the set $ \text{\Large $\mathit
m$}_{A, \varphi}(h) $ contains a unique probability measure.
\end{theorem}
 
\smallskip

\begin{theorem} \label{T16}
Let $ \varphi \in C^0(\Omega, \mathbb R^n) $ be a locally constant
function. Every point of the {interior of the rotation set $
\varphi_*(\mathcal M_\sigma) $} is the rotation vector of an ergodic
probability measure.
\end{theorem}
 
\smallskip

An example of a locally constant function is $\varphi= (I_{\overline{01}},\,\, 3\, I_{\overline{1,0}}, \,\,-4.3\, I_{\overline{11}}).$
 
\smallskip

\begin{theorem} \label{T17}
Suppose $\varphi \in C^0(\Omega, \mathbb Q) $ is a locally
constant function, and $ A $ is a H\"older potential. Consider a rational vector 
 $ r \in \text{int}(\varphi_*(\mathcal M_\sigma))$. Then,
$$ \beta_{A, \varphi}(r) = \sup \left \{\int A \; d\mu: \mu \in \varphi_*^{-1} (r),\, \mu \,\,\text{a  periodic probability measure} \right \}. $$
\end{theorem}

 \section{Changing the dynamics} 
 
 One can consider a different dynamics and  similar results are true.

Consider the transformation $T(x)= 2 x$ (mod 1) acting on $[0,1]$, and the set $\mathcal{M}_T$ of $T$-invariant probabilities. Given  a H\"{o}lder potential $A:[0,1] \to \mathbb{R}$, we say that $\mu^A \in \mathcal{M}_T$ is maximizing for $A$, if
$$\alpha(A) := \sup \{\int A\,  d \rho \, |\, \rho \in \mathcal{M}_T\}= \int A d \mu^A.$$

In a similar way as before we can say that a continuous function $u:[0,1]\to \mathbb{R}$ is a subaction for $A$ if

\begin{equation}\label{ditc12}  u(T(x)) -  u(x) - A(x) + \alpha(A) \geq 0.\end{equation}
By definition we  set $R$ as 
 \begin{equation}\label{x1} R(x)\,:=\, u(T(x)) -  u(x) - A(x) + \alpha(A) \geq 0,\end{equation}

One can show that for any point $x$ in the support of a maximizing measure for $A$, we have $R(x)= 0$. In the same way as before, we can take advantage of the  $1/2$ iterative procedure (using your desktop)  to get explicit results; which can be checked with pen and paper. In this case, the method is more direct than the case of the symbolic space; you do not need to use the binary expansion \eqref{soz} (see \cite{FLO1}).

\smallskip

We show in Figure  \ref{fig:jher}  the  graph of a potential  $A:[0,1] \to \mathbb{R}$ (which is linear by part), the graph of  the calibrated subaction $u$ (we get from the iterative procedure), and the graph of $R$.  It follows from a simple inspection analysis of the function $R$ that  the maximizing probability is  $\mu^A= \frac{1}{2} \delta_{1/3} +  \frac{1}{2} \delta_{2/3}$.

In Figure \ref{fig:jt} we exhibit one more application of the use of the $1/2$ iterative procedure taken from \cite{FeLoEl1} (see also \cite{ConGu}).

 In order to test its efficiency, in \cite{FeLoEl1} the method was applied to the potential $A(x)=- |x- K|$, where $K$ is the Cantor set in  $[0,1]$. In this case, a large number of iterations are needed to get a reasonable approximation of $u$ (producing   $x \to R(x)$, such that  $R(x)$ is very close to zero  for $x\in K$).

Some related references on Ergodic Optimization for one-dimensional maps are \cite{GSZ}, \cite{HXY} and \cite{RS}.

\begin{figure}[h]
  \centering
  \includegraphics[scale=0.27]{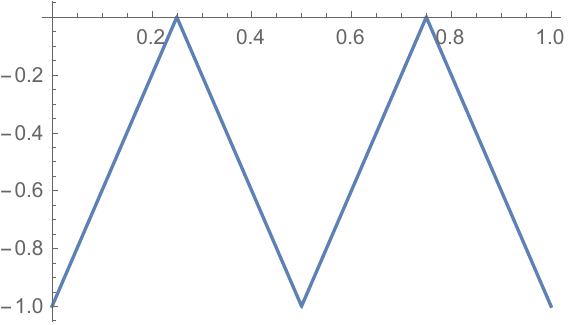}\qquad
  \includegraphics[scale=0.27]{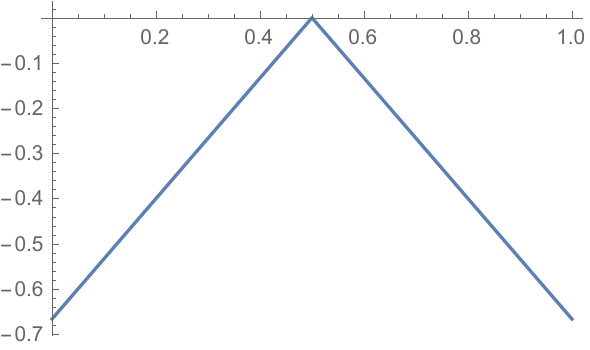}\qquad
  \includegraphics[scale=0.27]{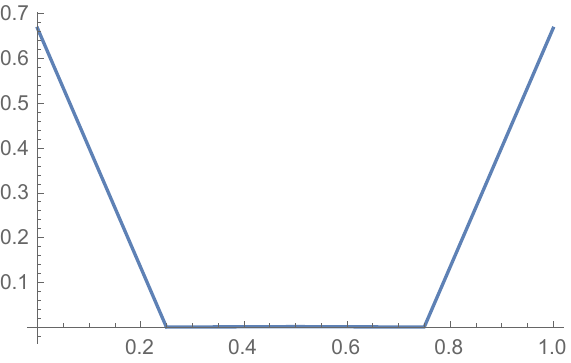}
  \caption{ From left to right: the  graph of the potential  $A$, the graph of  the calibrated subaction $u$ and the graph of $R$. The dynamics is given by the map $T(x)=2 x$ (mod 1) and  the maximizing probability is  $\mu^A= \frac{1}{2} \delta_{1/3} +  \frac{1}{2} \delta_{2/3}$.}\label{fig:jher}
\end{figure}

\begin{figure}[h]
  \centering
  \includegraphics[scale=0.10]{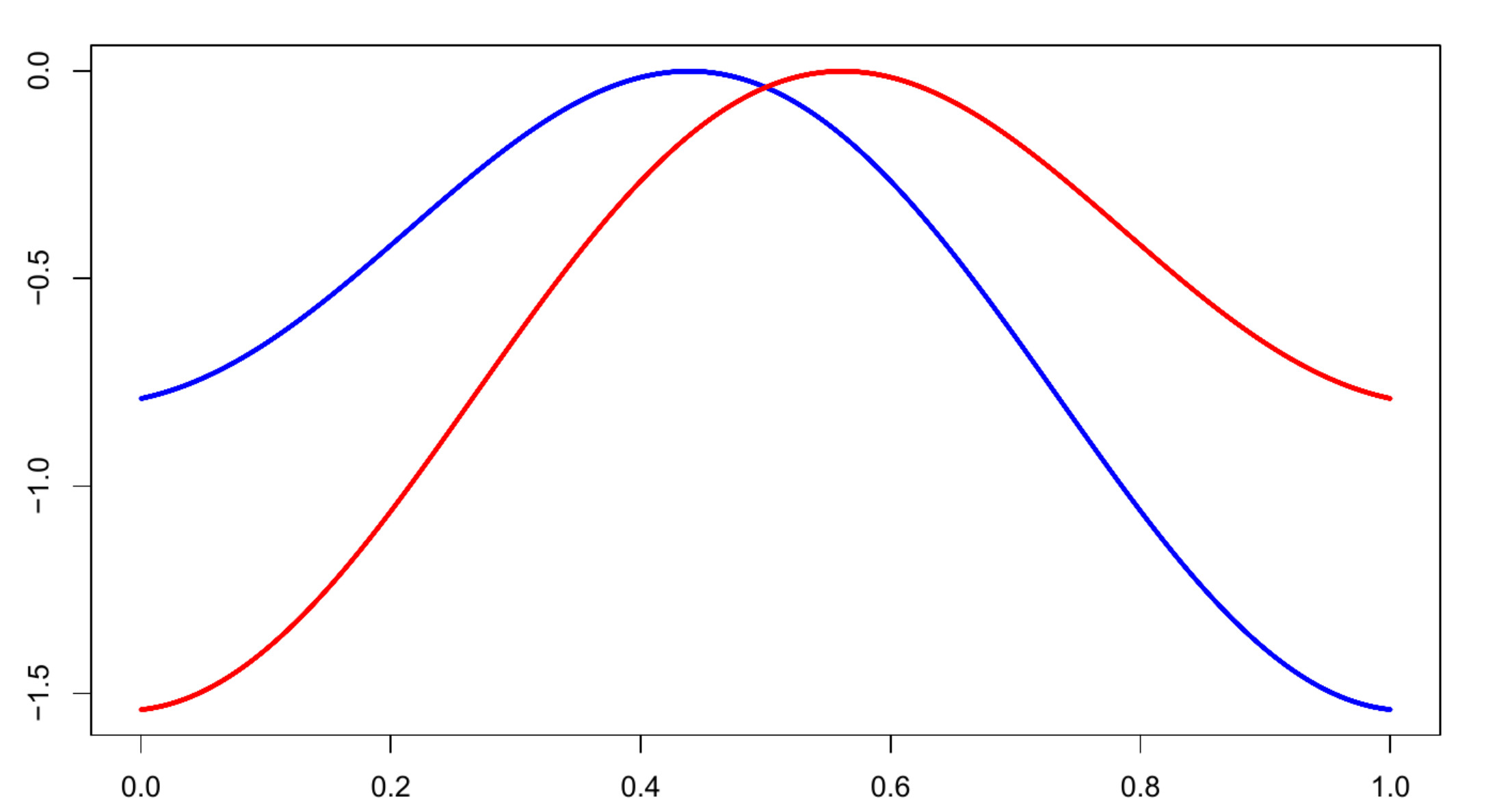}\qquad
  \includegraphics[scale=0.37]{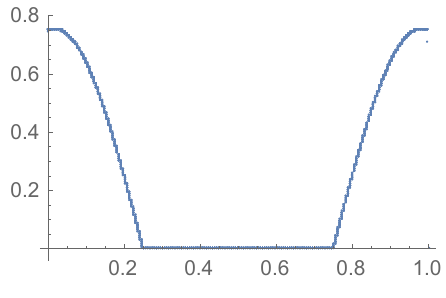}
  \caption{Case $A(x)=\sin^2\,(2 \pi x)$ - From the $1/2$ iterative procedure: on the left side, we show the approximated subaction $u$ which is given by the supremum of the two functions in red and in blue. The graph of $R$ using the approximation of the subaction $u$ is shown in the right-hand picture. The orbit of period 2 given by $\frac{1}{2} \delta_{1/3} + \frac{1}{2} \delta_{2/3}$ is inside the set $R=0$, and therefore is a maximizing probability $\mu^A$ for $A$. The expression of the two functions on the left side is given explicitly by a convergent series (see Section 5 in \cite{FeLoEl1}),  and  the use of the method shows a very good approximation of the analytical expression.}\label{fig:jt}
\end{figure}

\end{document}